\documentclass[a4paper,twoside,12pt]{article}
\usepackage[english]{babel}

\usepackage{amsmath,amsfonts,amssymb,amsthm}
\newtheorem{teo}{Theorem}

\newtheorem{remark}{Remark}

\newtheorem{ex}{Example}

 \title{{\bf  Evolution  Equations in    Hilbert Spaces via the Lacunae Method}}

\author{Maksim \,V.~Kukushkin   \\ \\
 \small  \textit{Moscow State University of Civil Engineering, 129337,  Moscow, Russia}\\
 \small\textit{Kabardino-Balkarian Scientific Center, RAS, 360051,  Nalchik, Russia}\\
\textit{\small\textit{kukushkinmv@rambler.ru}} }

\date{}
\begin{document}

\maketitle

\begin{abstract}
In this paper we consider evolution equations in the abstract Hilbert space under the special conditions imposed on the operator at  the right-hand side of the equation. We establish the method that allows us to  formulate the  existence and uniqueness  theorem and find a solution in the form of a series on the root vectors of the right-hand side. We consider fractional differential equations of various kinds as an application. Such operators as the Riemann-Liouville fractional differential operator,  the Riesz potential, the difference operator have been involved.

\end{abstract}
\begin{small}\textbf{Keywords:} Evolution equation; Fractional differential equations;
 Strictly accretive operator;  Abel-Lidskii basis property;   Schatten-von Neumann  class; convergence exponent.   \\\\
{\textbf{MSC} 34G25;
  47B28; 47A10; 47B12; 47B10;  26A33; 39A05}
\end{small}

\section{Introduction}

 In the paper \cite{firstab_lit:1kukushkin2021} we obtained the clarification of the results by Lidskii  V.B. \cite{firstab_lit:1Lidskii}  on  the decomposition on the root vector system of the non-selfadjoint operator. We used a technique of the entire function theory and introduce  a so-called  Schatten-von Neumann class of the convergence  exponent. Considering strictly accretive operators satisfying special conditions formulated in terms of the norm, we constructed a   sequence of contours of the power type  in the  contrary to the results by  Lidskii V.B.  \cite{firstab_lit:1Lidskii}, where a sequence of contours of the  exponential  type was used. In this paper we produce the   application of the mentioned method to evolution equations in the abstract Hilbert space  with the right-hand side of the special type. Here, we should  appeal to a plenty of applications to concrete differential equations  connected with  modeling  various physical -
chemical processes: filtration of liquid and gas in highly porous fractal   medium; heat exchange processes in medium  with fractal structure and memory; casual walks of a point particle that starts moving from the origin
by self-similar fractal set; oscillator motion under the action of
elastic forces which is  characteristic for  viscoelastic media, etc.
  In particular,  we would like   to  study  the existence and uniqueness theorems  for evolution equations   with the right-hand side -- a  differential operator  with a fractional derivative in  final terms. In this connection such operators as a Riemann-Liouville  fractional differential operator,    Kipriyanov operator, Riesz potential,  difference operator are involved.
Note that analysis of the required  conditions imposed upon the right-hand side of the studied class of evolution equations deserves to be mentioned. In this regard we should note  a well-known fact (see for instance  \cite{firstab_lit:Shkalikov A.})  that a particular interest appears in the case when a senior term of the operator
 (see \cite{firstab_lit(arXiv non-self)kukushkin2018})  is not selfadjoint at least for in the contrary case there is a plenty of results devoted to the topic within the framework of which  the following papers are well-known
\cite{firstab_lit:1Katsnelson},\cite{firstab_lit:1Krein},\cite{firstab_lit:Markus Matsaev},\cite{firstab_lit:2Markus},\cite{firstab_lit:Shkalikov A.}. Indeed, most of them deal with a decomposition of the  operator  to a sum  where the senior term
     must be either a selfadjoint or normal operator. In other cases the  methods of the papers
     \cite{kukushkin2019}, \cite{firstab_lit(arXiv non-self)kukushkin2018} become relevant   and allow us  to study spectral properties of   operators  whether we have the mentioned above  representation or not. Here, we ought to    stress that the results of  the papers \cite{firstab_lit:2Agranovich1994},\cite{firstab_lit:Markus Matsaev}  can be  also applied to study non-selfadjoin operators but  based on the sufficiently strong assumption regarding the numerical range of values of the operator (the numerical range belongs to a parabolic domain).
The methods of \cite{firstab_lit(arXiv non-self)kukushkin2018} that are   applicable to study    non-selfadjoint  operators can be  used in the natural way  if we deal with a more abstract construction -- the  infinitesimal generator of a semigroup of contraction \cite{kukushkin2021a}.  The  central challenge  of the latter  paper  is how  to create a model   representing     a  composition of  fractional differential operators   in terms of the semigroup theory. Here we should note that motivation arouse in connection with the fact that
  a second order differential operator can be presented  as a some kind of  a  transform of   the infinitesimal generator of a shift semigroup and stress that
  the eigenvalue problem for the operator
     was previously  studied by methods of  theory of functions   \cite{firstab_lit:1Nakhushev1977}, \cite{firstab_lit:1Aleroev1994}.
Having been inspired by   novelty of the  idea  we generalize a   differential operator with a fractional integro-differential composition  in the final terms   to some transform of the corresponding  infinitesimal generator of the shift semigroup.
By virtue of the   methods obtained in the paper
\cite{firstab_lit(arXiv non-self)kukushkin2018} we   managed  to  study spectral properties of the  infinitesimal generator  transform and obtained an outstanding result --
   asymptotic equivalence between   the
real component of the resolvent and the resolvent of the   real component of the operator. The relevance is based on the fact that
   the  asymptotic formula  for   the operator  real component  can be  established in most  cases due to well-known asymptotic relations  for the regular differential operators as well as for the singular ones
 \cite{firstab_lit:Rosenblum}. Thus,  we have theorems establishing spectral properties of some class of non-selfadjoint operators which allow  us, jointly with the results  \cite{firstab_lit:1kukushkin2021}, to study the Cauchy  problem for the evolution equation by the functional analysis methods. Note that the abstract approach to the Cauchy problem for the fractional evolution equation was previously implemented in the papers \cite{firstab_lit:Bazhl},\cite{Ph. Cl}. However, the  main advantage of this paper is the obtained formula for the solution of the evolution equation with the relatively wide conditions imposed upon the right-hand side,  wherein the derivative at the left-hand side is supposed to be of the real order.
 We consider the evolution equations with the right-hand side -- an operator function of the power  type. This problem appeals to many ones that lie  in the framework of the theory of differential equations for instance in the paper   \cite{L. Mor} the solution of the  evolution equation modeling the switching kinetics of ferroelectrics in the injection mode  can be obtained in the analytical way if we impose the conditions upon the right-hand side. The following papers deal with equations which can be studied by the obtained in this paper abstract method \cite{firstab_lit:Mainardi F.},\cite{firstab_lit:Mamchuev2017a}, \cite{firstab_lit:Mamchuev2017}, \cite{firstab_lit:Pskhu},       \cite{firstab_lit:Wyss}. Thus, we can claim that  the offered approach is undoubtedly novel and relevant.

\section{Preliminaries}

Let    $ C,C_{i} ,\;i\in \mathbb{N}_{0}$ be   real positive constants. We   assume   that  a  value of $C$    can be different in   various formulas and parts of formulas  but   values of $C_{i} $ are  certain. Denote by $ \mathrm{Fr}\,M$   the set of boundary points of the set $M.$    Everywhere further, if the contrary is not stated, we consider   linear    densely defined operators acting on a separable complex  Hilbert space $\mathfrak{H}$. Denote by $ \mathcal{B} (\mathfrak{H})$    the set of linear bounded operators   on    $\mathfrak{H}.$  Denote by
    $\tilde{L}$   the  closure of an  operator $L.$ We establish the following agreement on using  symbols $\tilde{L}^{i}:= (\tilde{L})^{i},$ where $i$ is an arbitrary symbol.  Denote by    $    \mathrm{D}   (L),\,   \mathrm{R}   (L),\,\mathrm{N}(L)$      the  {\it domain of definition}, the {\it range},  and the {\it kernel} or {\it null space}  of an  operator $L$ respectively. The deficiency (codimension) of $\mathrm{R}(L),$ dimension of $\mathrm{N}(L)$ are denoted by $\mathrm{def}\, L,\;\mathrm{nul}\,L$ respectively. Assume that $L$ is a closed   operator acting on $\mathfrak{H},\,\mathrm{N}(L)=0,$  let us define a Hilbert space
$
 \mathfrak{H}_{L}:= \big \{f,g\in \mathrm{D}(L),\,(f,g)_{ \mathfrak{H}_{L}}=(Lf,Lg)_{\mathfrak{H} } \big\}.
$
Consider a pair of complex Hilbert spaces $\mathfrak{H},\mathfrak{H}_{+},$ the notation
$
\mathfrak{H}_{+}\subset\subset\mathfrak{ H}
$
   means that $\mathfrak{H}_{+}$ is dense in $\mathfrak{H}$ as a set of    elements and we have a bounded embedding provided by the inequality
$
\|f\|_{\mathfrak{H}}\leq C_{0}\|f\|_{\mathfrak{H}_{+}},\,C_{0}>0,\;f\in \mathfrak{H}_{+},
$
moreover   any  bounded  set with respect to the norm $\mathfrak{H}_{+}$ is compact with respect to the norm $\mathfrak{H}.$
  Let $L$ be a closed operator, for any closable operator $S$ such that
$\tilde{S} = L,$ its domain $\mathrm{D} (S)$ will be called a core of $L.$ Denote by $\mathrm{D}_{0}(L)$ a core of a closeable operator $L.$ Let    $\mathrm{P}(L)$ be  the resolvent set of an operator $L$ and
     $ R_{L}(\zeta),\,\zeta\in \mathrm{P}(L),\,[R_{L} :=R_{L}(0)]$ denotes      the resolvent of an  operator $L.$ Denote by $\lambda_{i}(L),\,i\in \mathbb{N} $ the eigenvalues of an operator $L.$
 Suppose $L$ is  a compact operator and  $N:=(L^{\ast}L)^{1/2},\,r(N):={\rm dim}\,  \mathrm{R}  (N);$ then   the eigenvalues of the operator $N$ are called   the {\it singular  numbers} ({\it s-numbers}) of the operator $L$ and are denoted by $s_{i}(L),\,i=1,\,2,...\,,r(N).$ If $r(N)<\infty,$ then we put by definition     $s_{i}=0,\,i=r(N)+1,2,...\,.$
 Let  $\nu(L)$ denotes   the sum of all algebraic multiplicities of an  operator $L.$ Denote by $n(r)$ a function equals to the quantity of the elements of the sequence $\{a_{n}\}_{1}^{\infty},\,|a_{n}|\uparrow\infty$ within the circle $|z|<r.$ Let $A$ be a compact operator, denote by $n_{A}(r)$   {\it counting function}   a function $n(r)$ corresponding to the sequence  $\{s^{-1}_{i}(A)\}_{1}^{\infty}.$
  Let  $\mathfrak{S}_{p}(\mathfrak{H}),\, 0< p<\infty $ be       a Schatten-von Neumann    class and      $\mathfrak{S}_{\infty}(\mathfrak{H})$ be the set of compact operators.
   Denote by $\tilde{\mathfrak{S}}_{\rho}(\mathfrak{H})$ the class of the operators such that
$
 A\in \tilde{\mathfrak{S}}_{\rho}(\mathfrak{H}) \Rightarrow\{A\in\mathfrak{S}_{\rho+\varepsilon},\,A \overline{\in} \,\mathfrak{S}_{\rho-\varepsilon},\,\forall\varepsilon>0 \}.
$
In accordance with \cite{firstab_lit:1kukushkin2021} we will call it   {\it Schatten-von Neumann    class of the convergence exponent}.
Suppose  $L$ is  an   operator with a compact resolvent and
$s_{n}(R_{L})\leq   C \,n^{-\mu},\,n\in \mathbb{N},\,0\leq\mu< \infty;$ then
 we
 denote by  $\mu(L) $   order of the     operator $L$  (see \cite{firstab_lit:Shkalikov A.}).
 Denote by  $ \mathfrak{Re} L  := \left(L+L^{*}\right)/2,\, \mathfrak{Im} L  := \left(L-L^{*}\right)/2 i$
  the  real  and   imaginary components    of an  operator $L$  respectively.
In accordance with  the terminology of the monograph  \cite{firstab_lit:kato1980} the set $\Theta(L):=\{z\in \mathbb{C}: z=(Lf,f)_{\mathfrak{H}},\,f\in  \mathrm{D} (L),\,\|f\|_{\mathfrak{H}}=1\}$ is called the  {\it numerical range}  of an   operator $L.$
  An  operator $L$ is called    {\it sectorial}    if its  numerical range   belongs to a  closed
sector     $\mathfrak{ L}_{\iota}(\theta):=\{\zeta:\,|\arg(\zeta-\iota)|\leq\theta<\pi/2\} ,$ where      $\iota$ is the vertex   and  $ \theta$ is the semi-angle of the sector   $\mathfrak{ L}_{\iota}(\theta).$ If we want to stress the  correspondence  between $\iota$ and $\theta,$  then   we will write $\theta_{\iota}.$
 An operator $L$ is called  {\it bounded from below}   if the following relation  holds  $\mathrm{Re}(Lf,f)_{\mathfrak{H}}\geq \gamma_{L}\|f\|^{2}_{\mathfrak{H}},\,f\in  \mathrm{D} (L),\,\gamma_{L}\in \mathbb{R},$  where $\gamma_{L}$ is called a lower bound of $L.$ An operator $L$ is called  {\it   accretive}   if  $\gamma_{L}=0.$
 An operator $L$ is called  {\it strictly  accretive}   if  $\gamma_{L}>0.$      An  operator $L$ is called    {\it m-accretive}     if the next relation  holds $(A+\zeta)^{-1}\in \mathcal{B}(\mathfrak{H}),\,\|(A+\zeta)^{-1}\| \leq   (\mathrm{Re}\zeta)^{-1},\,\mathrm{Re}\zeta>0. $
    An operator $L$ is called     {\it symmetric}     if one is densely defined and the following  equality  holds $(Lf,g)_{\mathfrak{H}}=(f,Lg)_{\mathfrak{H}},\,f,g\in   \mathrm{D}  (L).$  Consider a   sesquilinear form   $ t  [\cdot,\cdot]$ (see \cite{firstab_lit:kato1980} )
defined on a linear manifold  of the Hilbert space $\mathfrak{H}.$
Let   $  \mathfrak{h}=( t + t ^{\ast})/2,\, \mathfrak{k}   =( t - t ^{\ast})/2i$
   be a   real  and    imaginary component     of the   form $  t $ respectively, where $ t^{\ast}[u,v]=t \overline{[v,u]},\;\mathrm{D}(t ^{\ast})=\mathrm{D}(t).$ Denote by $   t  [\cdot ]$ the  quadratic form corresponding to the sesquilinear form $ t  [\cdot,\cdot].$ According to these definitions, we have $
 \mathfrak{h}[\cdot]=\mathrm{Re}\,t[\cdot],\,  \mathfrak{k}[\cdot]=\mathrm{Im}\,t[\cdot].$ Denote by $\tilde{t}$ the  closure   of a   form $t.$  The range of a quadratic form
  $ t [f],\,f\in \mathrm{D}(t),\,\|f\|_{\mathfrak{H}}=1$ is called    {\it range} of the sesquilinear form  $t $ and is denoted by $\Theta(t).$
 A  form $t$ is called    {\it sectorial}    if  its    range  belongs to   a sector  having  a vertex $\iota$  situated at the real axis and a semi-angle $0\leq\theta<\pi/2.$   Due to Theorem 2.7 \cite[p.323]{firstab_lit:kato1980}  there exist unique    m-sectorial operators  $T_{t},T_{ \mathfrak{h}} $  associated  with   the  closed sectorial   forms $t,  \mathfrak{h}$   respectively.   The operator  $T_{  \mathfrak{h}} $ is called  a {\it real part} of the operator $T_{t}$ and is denoted by  $Re\, T_{t}.$
Everywhere further,   unless  otherwise  stated,  we   use  notations of the papers   \cite{firstab_lit:1Gohberg1965},  \cite{firstab_lit:kato1980},  \cite{firstab_lit:kipriyanov1960}, \cite{firstab_lit:1kipriyanov1960},
\cite{firstab_lit:samko1987}.
Consider the following hypotheses regarding an operator.\\

 \noindent ($ \mathrm{H}1 $) There  exists a Hilbert space $\mathfrak{H}_{+}\subset\subset\mathfrak{ H}$ and a linear manifold $\mathfrak{M}$ that is  dense in  $\mathfrak{H}_{+}.$ The operator $L$ is defined on $\mathfrak{M}.$\\

 \noindent  $( \mathrm{H2} )  \,\left|(Lf,g)_{\mathfrak{H}}\right|\! \leq \! C_{1}\|f\|_{\mathfrak{H}_{+}}\|g\|_{\mathfrak{H}_{+}},\,
      \, \mathrm{Re}(Lf,f)_{\mathfrak{H}}\!\geq\! C_{2}\|f\|^{2}_{\mathfrak{H}_{+}} ,\,f,g\in  \mathfrak{M},\; C_{1},C_{2}>0.
$\\

\noindent Throughout the paper we consider a restriction   $ W $ of the operator  $L$    on the set $\mathfrak{M}.$    We also    use  the  short-hand  notations $A:=R_{\tilde{W}},\,\mu:=\mu(H),$ where $H:=Re \tilde{W}.$    \\

\noindent{\bf  Auxiliary propositions}\\

Firstly, we consider   general statements proved in \cite{firstab_lit:1kukushkin2021}
with the made refinement related to the involved notion of the convergence exponent as well as newly constructed sequence of contours allowing to arrange the eigenvectors in the power type way, we used this expression following  the literary style of the monograph \cite{firstab_lit:1Lidskii}.   We implement  the approach that refers us to the notion  -- operator order, it gives us an opportunity to reformulate  results of the spectral theory in the more convenient and applicable way.
Recall that in the paper \cite{firstab_lit:1Lidskii} there was considered a sequence of   contours of the exponential type, the condition $\alpha>\rho$ (here and further $\rho$ denotes the index of the Schatten-von Neumann  class of the convergence exponent) is imposed (see \cite{firstab_lit:1Lidskii}, \cite{firstab_lit:1kukushkin2021}).   We improved this result in the paper \cite{firstab_lit:1kukushkin2021} in the following  sense,   we produced a sequence of the power type contours what  gives us the opportunity to obtain a solution of the problem in the case $\alpha=\rho.$ Moreover, we have omitted the conditions imposed on the semi-angle of the sector containing the numerical range of values of the involved operator. Such a significant achievement is  obtained by virtue of the way of choosing  a contour   which we consider     throughout the paper
$
\gamma:= \mathrm{Fr}\left\{  \mathfrak{L}_{\iota}(\theta_{\iota}+\varepsilon)\cap \mathfrak{L}_{0}(\theta_{0}+\varepsilon)\setminus \mathfrak{M}_{r}\right\},\;\mathfrak{M}_{r}:=\left\{\lambda:\;|\lambda|<r,\,|\mathrm{arg} \lambda|\leq \theta_{0} \right\},
$
$
\iota=C_{2}(1-C_{1} \mathrm{ctg} \theta_{\iota}/C_{2}),
$
where  the semi-angle  $\theta_{\iota}$   related to the operator $W$     is   sufficiently small (see reasonings of Theorem 2 \cite{firstab_lit:1kukushkin2021}),   $r$ is chosen so that the   operator $ (E-\lambda A)^{-1}  $ is regular within the corresponding  circle,  $\varepsilon>0$ is    sufficiently small.
       The auxiliary theorems given bellow (see \cite{firstab_lit:1kukushkin2021}) give us a tool to study the existence and uniqueness theorems in the abstract Hilbert space. Moreover, we obtain a solution that  can be presented by the series on the operator $A$ root vectors  $e_{q_{\xi}+i} $ with the coefficients $c_{q_{\xi}+i},$   where the index $q$ relates to the eigenvalue, the index $\xi$ relates to the geometrical multiplicity, the index $i$ relates to the algebraic multiplicity, the convergence is understood in the Abel-Lidskii sense (see \cite{firstab_lit:1Lidskii}). The idea  of   the  proofs of the   auxiliary  theorems given bellow   belongs to Lidskii V.B. However, we produce the proofs in  \cite{firstab_lit:1kukushkin2021}, since the made  refinement corresponding to   the case, when $\rho$ does not equal  the index of the Schatten-von Neumann  class, deserves to be considered itself.

\newpage
\begin{teo}\label{T1} Assume that hypothesis $\mathrm{H}1,\mathrm{H}2$  hold,   $A\in\tilde{\mathfrak{S}}_{\rho},\,\rho\leq\alpha.$   Moreover  in the case $A  \in \tilde{\mathfrak{S}}_{\rho}\setminus  \mathfrak{S}_{\rho}$ the additional condition holds
$$
   \frac{n_{A^{m+1}}(r^{m+1})}{r^{\rho} }\rightarrow 0,\, m=[\rho].
$$
Then a sequence of natural numbers $\{N_{\nu}\}_{0}^{\infty}$ can be chosen so that
$$
\frac{1}{2\pi i}\int\limits_{\gamma}e^{-\lambda^{\alpha}t}A(E-\lambda A)^{-1}f d \lambda =   \sum\limits_{\nu=0}^{\infty} \sum\limits_{q=N_{\nu}+1}^{N_{\nu+1}}\sum\limits_{\xi=1}^{m(q)}\sum\limits_{i=0}^{k(q_{\xi})}e_{q_{\xi}+i}c_{q_{\xi}+i}(t) ,
$$
moreover
\begin{equation*}
  \sum\limits_{\nu=0}^{\infty}\left\|\sum\limits_{q=N_{\nu}+1}^{N_{\nu+1}}\sum\limits_{\xi=1}^{m(q)}
  \sum\limits_{i=0}^{k(q_{\xi})}e_{q_{\xi}+i}c_{q_{\xi}+i}(t)\right\|_{\mathfrak{H}}<\infty.
\end{equation*}
\end{teo}
\begin{teo}\label{T2}
Assume that hypotheses $\mathrm{H}1,\mathrm{H}2$ hold,    $\alpha>2/\mu,\,\mu\in (0,1]$ and $\alpha> 1,\,\mu\in (1,\infty).$
Then a sequence of the natural  numbers  $\{N_{\nu}\}_{0}^{\infty}$ can be chosen   so that
$$
 \frac{1}{2\pi i}\int\limits_{\gamma}e^{-\lambda^{\alpha}t}A(E-\lambda A)^{-1}f d \lambda
=  \sum\limits_{\nu=0}^{\infty}  \sum\limits_{q=N_{\nu}+1}^{N_{\nu+1}}\sum\limits_{\xi=1}^{m(q)}\sum\limits_{i=0}^{k(q_{\xi})}e_{q_{\xi}+i}c_{q_{\xi}+i}(t),
$$
where
\begin{equation*}
  \sum\limits_{\nu=0}^{\infty}\left\|\sum\limits_{q=N_{\nu}+1}^{N_{\nu+1}}\sum\limits_{\xi=1}^{m(q)}\sum\limits_{i=0}^{k(q_{\xi})}e_{q_{\xi}+i}c_{q_{\xi}+i}(t)\right\|_{\mathfrak{H}}<\infty,
\end{equation*}
the following relation holds  for the eigenvalues
$$
|\lambda_{N_{\nu}+k}|-|\lambda_{N_{\nu}+k-1}|\leq C |\lambda_{N_{\nu}+k}|^{1-1/\tau},\;
 k=2,3,...,N_{\nu+1}-N_{\nu},\;0<\tau<\mu.
$$
\end{teo}
\begin{teo}\label{T3}
Assume that  a normal operator satisfies  the    hypotheses $\mathrm{H}1,\mathrm{H}2,$       $\alpha>1,$   the condition  $
 (\ln^{1+1/\alpha}x)'_{\lambda_{i}(H)}  =o(  i^{-1/\alpha}   )
$
holds. Then a sequence of the natural  numbers  $\{N_{\nu}\}_{0}^{\infty}$ can be chosen  so that
\begin{equation*}\label{2a}
 \frac{1}{2\pi i}\int\limits_{\gamma}e^{-\lambda^{\alpha}t}A(E-\lambda A)^{-1}f d \lambda
=  \sum\limits_{\nu=0}^{\infty}  \sum\limits_{q=N_{\nu}+1}^{N_{\nu+1}}\sum\limits_{\xi=1}^{m(q)}\sum\limits_{i=0}^{k(q_{\xi})}e_{q_{\xi}+i}c_{q_{\xi}+i}(t),
\end{equation*}
moreover
\begin{equation*}\label{13b1}
  \sum\limits_{\nu=0}^{\infty}\left\|\sum\limits_{q=N_{\nu}+1}^{N_{\nu+1}}
  \sum\limits_{\xi=1}^{m(q)}\sum\limits_{i=0}^{k(q_{\xi})}e_{q_{\xi}+i}c_{q_{\xi}+i}(t)\right\|_{\mathfrak{H}}<\infty,
\end{equation*}
 the following relation holds for the corresponding  eigenvalues
$$
|\lambda_{N_{\nu}+k}|-|\lambda_{N_{\nu}+k-1}|\leq C |\lambda_{N_{\nu}+k}|^{1-1/\tau},\;
 k=2,3,...,N_{\nu+1}-N_{\nu},\;0<\tau<1/\alpha.
$$
\end{teo}

\section{Main results}

In this section we consider evolution  equations in the abstract Hilbert space.
Having used an abstract theorem     formulated    in terms of the operator order,
we produce an example of the class of differential equations for which the made refinement  regarding the convergence  exponent  is relevant. More precisely,   under the assumption $\rho=\alpha$  the sequence of contours  may be chosen in a concrete -- power type  way,    what  provides a peculiar  validity of the  statement.
We prove the  existence and uniqueness theorem       and supply it with a plenty of applications.    We  consider   applications  to the differential equations in the concrete Hilbert spaces and involve such operators as Riemann-Liouville operator, Kipriyanov operator,    Riesz potential, difference operator.  Moreover, we produce  the artificially constructed normal operator for which the clarification of the Lidskii V.B. results  relevantly works.
 Further, we will consider a Hilbert space $\mathfrak{H}$ which consists of   element-functions $u:\mathbb{R}_{+}\rightarrow \mathfrak{H},\,u:=u(t),\,t\geq0$    and we will assume that if $u$ belongs to $\mathfrak{H}$    then the fact  holds for all values of the variable $t.$ Notice that under such an assumption all standard topological  properties as completeness, compactness etc remain correctly defined.   We understand such operations as differentiation and integration in the generalized sense that is caused by the topology of the Hilbert space $\mathfrak{H},$  more detailed information can be found in the   Chapter 4  \cite{firstab_lit:Krasnoselskii M.A.}.
   Consider a Cauchy problem
\begin{equation}\label{1}
  \frac{du}{dt}=-\tilde{W}^{n}u,\;u(0)=h\in \mathrm{D}(\tilde{W}),\;n=1,2,...\,.
\end{equation}
In the case when $\tilde{W}^{n}$ is accretive, we can suppose $h\in \mathfrak{H}$
(here we should note that the case $n=1$ was considered by Lidskii V.B. \cite{firstab_lit:1Lidskii}). Bellow, we formulate a result   which follows from the auxiliary theorems given above.
 \begin{teo}\label{T4}
Assume that  the conditions of one of the   Theorems \ref{T1}, \ref{T2}, \ref{T3}, hold      under the assumption $\alpha=n,$   then
there exists a solution of the Cauchy problem \eqref{1} in the form
\begin{equation}\label{2}
u(t)= \frac{1}{2\pi i}\int\limits_{\gamma}e^{-\lambda^{n}t}A(E-\lambda A)^{-1}h d \lambda
=  \sum\limits_{\nu=0}^{\infty}  \sum\limits_{q=N_{\nu}+1}^{N_{\nu+1}}\sum\limits_{\xi=1}^{m(q)}\sum\limits_{i=0}^{k(q_{\xi})}e_{q_{\xi}+i}c_{q_{\xi}+i}(t),
\end{equation}
where
\begin{equation*}
  \sum\limits_{\nu=0}^{\infty}\left\|\sum\limits_{q=N_{\nu}+1}^{N_{\nu+1}}\sum\limits_{\xi=1}^{m(q)}\sum\limits_{i=0}^{k(q_{\xi})}e_{q_{\xi}+i}c_{q_{\xi}+i}(t)\right\|_{\mathfrak{H}}<\infty,
\end{equation*}
a sequence of natural numbers $\{N_{\nu}\}_{0}^{\infty}$ can be chosen in accordance with the claim of the  corresponding theorem.
Moreover, the existing solution is unique, if the operator $\tilde{W}^{n}$ is accretive.
\end{teo}
\begin{proof}
Let us    find a solution of problem \eqref{1}  in the form   \eqref{2}.
We need  prove that  the following integral   converges   i.e.
\begin{equation}\label{3a}
 \frac{1}{2\pi i} \int\limits_{\gamma}e^{-\lambda^{n}t} (E-\lambda A)^{-1}h \,d\lambda \in \mathfrak{H},\;h\in \mathfrak{H}.
\end{equation}
For this purpose  consider a contour $\gamma_{k}:=\{\lambda\in \gamma,\,|\lambda|<R_{k},\,k\in \mathbb{N}\},\,R_{k}\uparrow\infty.$ Using simple estimating, then applying Lemma 6 \cite{firstab_lit:1kukushkin2021}, we get
$$
 \left\| \int\limits_{\gamma_{k}}e^{-\lambda^{n}t} (E-\lambda A)^{-1}h d\lambda \right\|_{\mathfrak{H}}
 \leq  \int\limits_{\gamma_{k}}|e^{-\lambda^{n}t}| \cdot \|(E-\lambda A)^{-1}h  \|_{\mathfrak{H}} |d\lambda| \leq C \|h\|_{\mathfrak{H}} \int\limits_{\gamma_{k}} e^{-t \mathrm{Re}\lambda^{n}}   |d\lambda|.
$$
It is clear that
$$
\int\limits_{\gamma_{k}} e^{- t\mathrm{Re}\lambda^{n}}   |d\lambda|\rightarrow C,\,k\rightarrow\infty.
$$
The latter fact gives us the desired result.
Since $A$ is bounded, then  we have
$$
\tilde{W}u(t)=\tilde{W} \left(\frac{1}{2\pi i} \int\limits_{\gamma}e^{-\lambda^{n}t} A(E-\lambda A)^{-1}h \,d\lambda\right)=\tilde{W}A \left(\frac{1}{2\pi i} \int\limits_{\gamma}e^{-\lambda^{n}t} (E-\lambda A)^{-1}h \,d\lambda\right)=
$$
$$
 =  \frac{1}{2\pi i} \int\limits_{\gamma}e^{-\lambda^{n}t} (E-\lambda A)^{-1}h \,d\lambda.
$$
Combining the latter relation with \eqref{3a}, we obtain    $u\in \mathrm{D}(\tilde{W}).$
Analogously to the above,  using   Lemma 6 \cite{firstab_lit:1kukushkin2021}, we can show that the the following derivative exists i.e.
\begin{equation*}
\frac{d u}{d t}=-\frac{1}{2\pi i} \int\limits_{\gamma}e^{-\lambda^{n}t}  \lambda^{n } A(E-\lambda A)^{-1}h \,d\lambda \in \mathfrak{H}.
\end{equation*}
 Notice  that $\lambda^{n} A^{n}(E-\lambda A)^{-1}=(E-\lambda A)^{-1}-(E+\lambda A+...+\lambda^{n-1}A^{n-1}),$
substituting this relation to the above formula,   we obtain
$$
A^{n-1}\frac{d u}{d t}=-\frac{1}{2\pi i} \int\limits_{\gamma}e^{-\lambda^{n}t}   (E-\lambda A)^{-1}h\, d\lambda+\frac{1}{2\pi i} \int\limits_{\gamma}e^{-\lambda^{n}t} \sum\limits_{k=0}^{n-1}\lambda^{k}A^{k}  h \,d\lambda.
$$
The second integral equals zero by virtue  of  the fact  that the   function under the integral is analytical inside the intersection of the  domain $G$  with the circle of  the arbitrary radius $R$   and it   decreases sufficiently fast on the arch of the radius $R,$ when $R\rightarrow\infty,$  here we denote  by $G$ the interior of the contour $\gamma.$      Thus, we have come to the relation
$$
A^{n-1}\frac{d u}{d t}=-\frac{1}{2\pi i} \int\limits_{\gamma}e^{-\lambda^{n}t}   (E-\lambda A)^{-1}h\, d\lambda .
$$
Since the left-hand side of the latter relation belongs to $\mathrm{D}(\tilde{W}^{n-1}),$ then we can claim that it is so  for the right-hand side also. It follows that $u\in \mathrm{D}(W^{n}).$ Now, if we recall the expression for $u,$ we get
$
A^{n-1}   u'_{t}  +\tilde{W }u=0.
$
Applying  the operator  $\tilde{W }^{n-1}$ to  both sides of the latter relation,  we obtain the fact that $u$ is a solution of  equation \eqref{1}. Let us show that the initial condition holds in the sense
$
u(t)   \xrightarrow[   ]{\mathfrak{H}}  h,\,t\rightarrow+0.
$
It becomes clear in the case  $h\in \mathrm{D}(\tilde{W})$ for in this case  it suffices to apply Lemma 7   \cite{firstab_lit:1kukushkin2021}, what  gives  us the desired result i.e. we can    put $u(0)=h.$  Consider a case when $h$ is an arbitrary element of the Hilbert space $\mathfrak{H}$  and let us involve the accretive property of the operator $\tilde{W}^{n}.$
Consider an operator
$$
S_{t}h=\frac{1}{2\pi i} \int\limits_{\gamma}e^{-\lambda^{n}t}A(E-\lambda A)^{-1}h\, d\lambda,\,t>0.
$$
In accordance with the above, it is clear that  $S_{t}:\mathfrak{H}\rightarrow \mathfrak{H}.$
Let us prove  that
$
\|S_{t}\|_{\mathfrak{H}\rightarrow \mathfrak{H}}\leq1,\;t>0.
$
Firstly, assume   that $h\in \mathrm{D}(\tilde{W}).$
Let us multiply the both sides of   relation \eqref{1} on $u$  in the sense of the inner product, we get
$
\left(u'_{t},u\right)_{\mathfrak{H}}+(\tilde{W}^{n}u,u)_{\mathfrak{H}}=0.
$
Consider a real part of the latter relation, we have
$
\mathrm{Re}\left(u'_{t},u\right)_{\mathfrak{H}}+\mathrm{Re}(\tilde{W}^{n}u,u)_{\mathfrak{H}}= \left(u'_{t},u\right)_{\mathfrak{H}}/2+ \left(u, u'_{t}\right)_{\mathfrak{H}}/2+\mathrm{Re}(\tilde{W}^{n}u,u)_{\mathfrak{H}}.
$
Therefore
$
   \left(\|u(t)\|_{\mathfrak{H}}^{2}\right)'_{t}  =-2\mathrm{Re}(\tilde{W}^{n}u,u)_{\mathfrak{H}}\leq 0.
$
Integrating both sides, we get
$$
  \|u(\tau)\|_{\mathfrak{H}}^{2}-  \|u(0)\|_{\mathfrak{H}}^{2}=\int\limits_{0}^{\tau} \frac{d }{dt}\|u(t)\|_{\mathfrak{H}}^{2} dt\leq 0.
$$
The last relation can be rewritten in the form
$
\|S_{t}h\|_{\mathfrak{H}}\leq \|h\|_{\mathfrak{H}},\,h\in \mathrm{D}(\tilde{W}).
$
Since $\mathrm{D}(\tilde{W})$ is a dense set in $ \mathfrak{H},$ then we obviously  obtain  the  desired result i.e. $\|S_{t}\|_{\mathfrak{H}\rightarrow \mathfrak{H}}\leq 1.$
 Now, having assumed that
$
h_{n}   \xrightarrow[   ]{\mathfrak{H}}  h,\,n\rightarrow \infty,\;\{h_{n}\}\subset \mathrm{D}(\tilde{W}),\,h\in \mathfrak{H},
$
consider the following reasonings
$
\|u(t)-h\|_{\mathfrak{H}}=\|S_{t}h-h\|_{\mathfrak{H}}=\|S_{t}h-S_{t}h_{n}+S_{t}h_{n}-h_{n}+h_{n}-    h\|_{\mathfrak{H}}\leq \|S_{t}\|\cdot\|h- h_{n}\|_{\mathfrak{H}}+\|S_{t}h_{n}-h_{n}\|_{\mathfrak{H}}+\|h_{n}-    h\|_{\mathfrak{H}}.
$
Note that
$
S_{t}h_{n}   \xrightarrow[   ]{\mathfrak{H}}  h_{n},\,t\rightarrow+0.
$
It is clear that  if we chose $n$ so that  $\|h- h_{n}\|_{\mathfrak{H}}<\varepsilon/3$ and  after that  chose $t$   so that $\|S_{t}h_{n}-h_{n}\|_{\mathfrak{H}}<\varepsilon/3,$ then we obtain
 $\forall\varepsilon>0,\,\exists \delta(\varepsilon):\,\|u(t)-h\|_{\mathfrak{H}}<\varepsilon,\,t<\delta.$
   Thus, we can put  $u(0)=h$ and claim that   the initial condition holds in the case $h\in \mathfrak{H}.$   The decomposition   on  the series  of the root vectors \eqref{2} is given by virtue of    Theorems \ref{T1}, \ref{T2}, \ref{T3} respectively.   The uniqueness follows easily from the fact that $\tilde{W}^{n}$ is accretive.  In this case, repeating the previous reasonings we come to
\begin{equation}\label{3}
  \|\phi(\tau)\|_{\mathfrak{H}}^{2}-  \|\phi(0)\|_{\mathfrak{H}}^{2}=\int\limits_{0}^{\tau} \frac{d }{dt}\|\phi(t)\|^{2} dt\leq 0,
\end{equation}
where $\phi $ is a sum of    solutions $u_{1} $ and $u_{2}.$ Notice that  by virtue of the initial conditions, we have  $\phi(0)=0,$  thus    relation \eqref{3} can hold only if $\phi =0.$  The proof is complete.
\end{proof}
\begin{remark}\label{R1} Note that the  assumption $ n>\rho ,$ that is   additional  to the ones of the above theorems, guaranies an opportunity to chose
 a sequence of numbers  $\{N_{\nu}\}_{0}^{\infty}$     so that in the used terms, we have
 $$
|\lambda_{N_{\nu}+k}|-|\lambda_{N_{\nu}+k-1}|\leq  e^{-|\lambda_{N_{\nu}}|^{\tau}} ,\;0<\tau<n-\rho.
$$
This fact follows from the results by Lidskii  V.B. \cite[p.23]{firstab_lit:1Lidskii}.
\end{remark}

\noindent{\bf Evolution  equations with the quasi-polynomial right-hand side }\\

Let $I:=(a,b)\subset \mathbb{R},\,\Omega:=[0,\infty),$ consider the functions  $u(t,x),\,t\in \Omega,\,x\in \bar{I}.$ In accordance with the said above, we will consider functional spaces with respect to the variable $x$ and we will assume that if $u$ belongs to a functional space then the fact  holds for all values of the variable $t.$
Consider a Cauchy problem
\begin{equation}\label{4}
     \frac{d u}{dt}=\sum\limits_{k=1}^{s} C_{k} D^{ k \vartheta}_{a+}u=:P_{s,\vartheta} u ,\; \vartheta>0,\;
       u(0,x)=h(x)\in L_{2}(I),
\end{equation}
where at the right-hand side we have a linear combination of the Riemann-Liouville fractional differential operators   acting in $L_{2}(I)$ with respect to the variable $x.$
 We will call (analogously to the theory of ordinary differential equations) the right-hand side of   equation \eqref{4} {\it quasi-polynomial}. Consider a case when the right-hand side of   equation \eqref{4} can be represented as follows
\begin{equation}\label{5}
P_{s,\vartheta} u= - (\eta D^{2}+\xi D^{\beta}_{a+})^{n}u,\;0<\beta<1/n,\,\eta<0,\,\xi >0.
\end{equation}
Let us show that under such assumptions   problem \eqref{4}  has a unique solution which can be found due to a certain formula.
Consider the operator
$
W:=\eta D^{2}+\xi D^{\beta}_{a+} ,\,\mathrm{D}(W)=C_{0}^{\infty}(I).
$
Note that the hypotheses  $\mathrm{H}1,\mathrm{H}2$  hold regarding the operator,  if we assume that $\mathfrak{H}:=L_{2}(I),\,\mathfrak{H}_{+}:=H^{1}_{0}(I).$  It follows from the strictly accretive property of the fractional differential operator (see \cite{kukushkin2019}) and the estimate
\begin{equation}\label{6}
\|D^{\beta}_{a+}f\|_{L_{2}}\leq C \|f\|_{H^{1}_{0}},\;f\in C_{0}^{\infty}(I),\,\beta\in(0,1).
\end{equation}
Let us prove that
\begin{equation}\label{7}
 -C (D^{2}f,f)_{L_{2}}\leq\mathrm{Re}(Wf,f)_{L_{2}}\leq -C (D^{2}f,f)_{L_{2}},\;f\in C_{0}^{\infty}(I).
\end{equation}
 Using    relation \eqref{6} and the Friedrichs inequality, we obtain
$
\mathrm{Re}(D^{\beta}_{a+}f,f)_{L_{2}}\leq\|f\|_{L_{2}}\|D^{\beta}_{a+}f\|_{L_{2}}\leq C\|f\|^{2}_{H_{0}^{1}}=-C(D^{2}f,f)_{L_{2}},\,f\in C_{0}^{\infty}(I),
$
  what gives us  the upper estimate.
The lower estimate follows easily from the  accretive   property of the fractional differential operator of the order less than one.
 Using   relation \eqref{7}, the corollary  of the minimax principle, we get
$
-\lambda_{j}(H)\asymp \lambda_{j}(D^{2}),
$
where  $ H$  is a real part   of the operator $\tilde{W}.$ Therefore, taking into account the well-known fact
$
 \lambda_{j}(D^{2})= -\pi^{2}j^{2}/(b-a)^{2},
$
we get
$
\lambda_{j}(H)\asymp j^{2},
$
it follows that   $\mu(H)>1.$
Now we can  study   the Cauchy  problem \eqref{4}  by restricting  the one  to   the problem
\begin{equation}\label{8}
\frac{du}{dt} =   -(\eta D^{2}+\xi D^{\beta }_{a+})^{n}u,\;u(0)=h\in \mathrm{D}(\tilde{W}).
\end{equation}
In accordance with   Theorem   \ref{T4},   we are able to present a solution of
the problem \eqref{4}, with the restricted assumptions regarding $h,$ as  follows
$$
 u(t)=\frac{1}{2 \pi i} \int\limits_{\gamma}e^{-\lambda^{n}t}  A (E-\lambda A )^{-1}hd\lambda,
$$
where the used terms   relate  to the operator    $\tilde{W} .$ Thus, we have in the reminder a question how  to weaken conditions imposed upon the function $h$ as well as wether the representation \eqref{5} holds. To answer the questions consider the following reasonings.
 Further, for the sake of the simplicity, we consider a case when $\eta =-1,\,\xi =1.$   This assumption does not restrict generality of reasonings.
Let us show that
$
\mathrm{D}(\tilde{W})\subset   H^{2}_{0}(I).
$
Using $\mathrm{H}2,$ we have the implication
$$
 f_{k}  \xrightarrow[      W       ]{}f \Longrightarrow  f_{k} \xrightarrow[              ]{H_{0}^{1}} f ,\;\{f_{k}\}_{1}^{\infty}\subset C_{0}^{\infty}(I).
$$
Applying  \eqref{6}, we get
$
  D^{\beta}_{a+}f_{k} \xrightarrow[              ]{L_{2}} D^{\beta}_{a+}f.
$
The following fact can be obtained easily, we have omitted the proof
$$
\{f_{k}  \xrightarrow[      W       ]{}f,\;  D^{\beta}_{a+}f_{k} \xrightarrow[              ]{L_{2}} D^{\beta}_{a+}f \}\Longrightarrow D^{2}f_{k}\xrightarrow[              ]{L_{2}}D^{2}f .
$$
Combining the above implications  we obtain the desired result i.e. $
\mathrm{D}(\tilde{W})\subset   H^{2}_{0}(I).
$
Consider a set
$
H_{0+}^{s}(I):=\{f:\,f\in H^{s}(I),\,f^{(k)}(a)=0,\,k=0,1,...,s-1\},\,s\in \mathbb{N}.
$
It is clear that $H^{s}_{0}(I)\subset H_{0+}^{s}(I),$ thus we can    define the operator $W_{+}$ as the extension  of the operator $\tilde{W}$ on the set $H_{0+}^{2}(I),$    we have $\tilde{W}\subset W_{+}.$  Let us show that
$\mathrm{D}(W^{n}_{+})=H_{0+}^{2n}(I).$ Assume that  $f\in H_{0+}^{2n}(I),$ then $f\in \mathrm{D}(W^{n}_{+}),$ it can be verified directly.
If $f\in \mathrm{D}(W^{n}_{+}),$ then in accordance with the definition, we have  $W^{n-1}_{+}f\in H_{0+}^{2}(I).$ It follows that
$
W_{+}g_{1}\in H_{0+}^{2}(I),$ where  $g_{1}=W^{n-2}_{+}f\in H_{0+}^{2}(I).
$
Hence
$
D^{2}g_{1}+D^{\beta}_{a+}g_{1}\in H_{0+}^{2}(I).
$
Applying the operator $I^{2}_{a+}$ to the both sides of the last relation, we easily  get
\begin{equation}\label{9}
 g_{1}+I^{2-\beta}_{a+}g_{1}\in H_{0+}^{4}(I).
\end{equation}
 Using the constructive features of relation \eqref{9}, we can conclude firstly  $g_{1}\in H_{0+}^{3}(I)$  and due to the same reasonings establish the fact
  $g_{1}\in H_{0+}^{4}(I)$ secondly, what gives us  $W^{n-2}_{+}f\in H_{0+}^{4}(I).$ Using the absolutely analogous reasonings we prove that
  $W^{n-k}_{+}f\in H_{0+}^{2k}(I),\;k= 1,2,...,n.$ Thus, we obtain the desired result.  Let us show that
\begin{equation}\label{10}
W^{n}_{+}f=\sum\limits_{k=0}^{n}(-1)^{n-k}C_{n}^{k} D^{\beta k+2(n-k)}_{a+}f,\;f\in H_{0+}^{2n}(I).
\end{equation}
We need  establish the formula
$
D^{\beta k}_{a+}D^{2(n-k)}f=D^{2(n-k)} D^{\beta k}_{a+}f= D_{a+}^{\beta k+2(n-k)} f,
 \,f\in H_{0+}^{2n}(I),\,
 (k= 1,2,...,n)
$
for this purpose, in accordance with Theorem 2.5   \cite[p.46]{firstab_lit:samko1987},  we should prove  that
$
f\in I^{\beta k +2(n-k)}_{a+}(L_{1})
$
or
$
f\stackrel{a.e.}=I^{2(n-k)}_{a+}\varphi,\;\varphi\in I^{\beta k }_{a+}(L_{1}).
$
We get $f\stackrel{a.e.}= I^{2n}_{a+}D^{2n}f = I^{2(n-k)}_{a+}D^{2(n-k)}f=   I^{2(n-k)}_{a+}\varphi,$ where $\varphi:=D^{2(n-k)}f.$
 Note that  the conditions of Theorem 13.2 \cite[p.229]{firstab_lit:samko1987} hold   i.e. the Marchaud derivative of the function $\varphi$ belongs to $L_{1}(I).$ Hence   $\varphi\in I^{\beta k }_{a+}(L_{1})$  and we obtain the required formula.
  Using the well-known   formulas for linear operators $(A+B)C\supseteq AC+BC,\,C(A+B)=CA+CB,$   applying the Leibniz formula, we obtain \eqref{10}.
Now, combining  the obvious inclusion  $\tilde{W}^{n}\subset W^{n}_{+}$ with  \eqref{10},  we get
$$
\tilde{W}^{n}\subset \sum\limits_{k=0}^{n}(-1)^{n-k}C_{n}^{k} D^{\beta k+2(n-k)}_{a+}.
$$
The next question is wether the operator $\tilde{W}^{n}$ is accretive.  By direct calculation, we have
$$
\mathrm{Re}(\tilde{W}^{n}f,f)_{L_{2}}= \sum\limits_{k=0}^{n} C_{n}^{k}  \mathrm{Re}\left( D^{ \beta k+(n-k)}_{a+}f,D^{n-k}f   \right)_{L_{2}}=
 \sum\limits_{k=0}^{n} C_{n}^{k}  \mathrm{Re}\left(D_{a+}^{\beta k}  g_{k},g_{k}   \right)_{L_{2}}\!\! \geq 0,
$$
where $g_{k}:=D^{n-k}f,\,f\in \mathrm{D}(\tilde{W}^{n}).$
Note that the last inequality holds by virtue of the strictly accretive property of the fractional differential operator of the order less than one (see \cite{kukushkin2019}).
 Thus, the uniqueness and the opportunity to weaken conditions imposed on $h$  follow    from Theorem  \ref{T4}. Here we should remark that the latter theorem  gives us the fact that the existing solution is unique in the set $\mathrm{D}(\tilde{W}^{n}),$ but using the same method we can   establish   the  uniqueness of the solution  of   problem  \eqref{4}.   Having known the root vectors of the operators $A,$ applying  formula \eqref{2}, we can represent the obtained solution as a series.\\

\noindent{\bf Kipriyanov operator}\\

Using     notations of the paper     \cite{firstab_lit:kipriyanov1960} we assume that $\Omega$ is a  convex domain of the  $m$ -  dimensional Euclidean space $\mathbb{E}^{m}$, $P$ is a fixed point of the boundary $\partial\Omega,$
$Q(r,\mathbf{e})$ is an arbitrary point of $\Omega;$ we denote by $\mathbf{e}$   a unit vector having a direction from  $P$ to $Q,$ denote by $r=|P-Q|$   the Euclidean distance between the points $P,Q,$ and   use the shorthand notation    $T:=P+\mathbf{e}t,\,t\in \mathbb{R}.$
We   consider the Lebesgue  classes   $L_{p}(\Omega),\;1\leq p<\infty $ of  complex valued functions.  For the function $f\in L_{p}(\Omega),$    we have
\begin{equation}\label{11}
\int\limits_{\Omega}|f(Q)|^{p}dQ=\int\limits_{\omega}d\chi\int\limits_{0}^{d(\mathbf{e})}|f(Q)|^{p}r^{m-1}dr<\infty,
\end{equation}
where $d\chi$   is an element of   solid angle of
the unit sphere  surface (the unit sphere belongs to $\mathbb{E}^{m}$)  and $\omega$ is a  surface of this sphere,   $d:=d(\mathbf{e})$  is the  length of the  segment of the  ray going from the point $P$ in the direction
$\mathbf{e}$ within the domain $\Omega.$
Without lose of   generality, we consider only those directions of $\mathbf{e}$ for which the inner integral on the right-hand  side of equality \eqref{11} exists and is finite. It is  the well-known fact that  these are almost all directions.
 We use a shorthand  notation  $P\cdot Q=P^{i}Q_{i}=\sum^{m}_{i=1}P_{i}Q_{i}$ for the inner product of the points $P=(P_{1},P_{2},...,P_{m}),\,Q=(Q_{1},Q_{2},...,Q_{m})$ which     belong to  $\mathbb{E}^{m}.$
     Denote by  $D_{i}f$  a distributional derivative of the function $f$ with respect to a coordinate variable with index   $1\leq i\leq m.$

Here, we study a case $\beta \in (0,1).$ Assume that  $\Omega\subset \mathbb{E}^{m}$ is  a convex domain, with a sufficient smooth boundary ($ C ^{3}$ class)   of the m-dimensional Euclidian space. For the sake of the simplicity we consider that $\Omega$ is bounded, but  the results  can be extended     to some type of    unbounded domains.
In accordance with the definition given in  the paper  \cite{firstab_lit:1kukushkin2018}, we consider the directional  fractional integrals.  By definition, put
$$
 (\mathfrak{I}^{\beta }_{0+}f)(Q  ):=\frac{1}{\Gamma(\beta )} \int\limits^{r}_{0}\frac{f (P+t \mathbf{e} )}{( r-t)^{1-\beta }}\left(\frac{t}{r}\right)^{m-1}\!\!\!\!dt,\,(\mathfrak{I}^{\beta }_{d-}f)(Q  ):=\frac{1}{\Gamma(\beta )} \int\limits_{r}^{d }\frac{f (P+t\mathbf{e})}{(t-r)^{1-\beta }}\,dt,
$$
$$
\;f\in L_{p}(\Omega),\;1\leq p\leq\infty.
$$
Also,     we   consider auxiliary operators,   the so-called   truncated directional  fractional derivatives    (see \cite{firstab_lit:1kukushkin2018}).  By definition, put
 \begin{equation*}
 ( \mathfrak{D} ^{\beta}_{0+,\,\varepsilon}f)(Q)=\frac{\beta }{\Gamma(1-\beta )}\int\limits_{0}^{r-\varepsilon }\frac{ f (Q)r^{m-1}- f(P+\mathbf{e}t)t^{m-1}}{(  r-t)^{\beta  +1}r^{m-1}}   dt+\frac{f(Q)}{\Gamma(1-\beta )} r ^{-\beta },\;\varepsilon\leq r\leq d ,
 $$
 $$
 (\mathfrak{D}^{\beta }_{0+,\,\varepsilon}f)(Q)=  \frac{f(Q)}{\varepsilon^{\beta }}  ,\; 0\leq r <\varepsilon ;
\end{equation*}
\begin{equation*}
 ( \mathfrak{D }^{\beta }_{d-,\,\varepsilon}f)(Q)=\frac{\beta }{\Gamma(1-\beta )}\int\limits_{r+\varepsilon }^{d }\frac{ f (Q)- f(P+\mathbf{e}t)}{( t-r)^{\beta  +1}} dt
 +\frac{f(Q)}{\Gamma(1-\beta )}(d-r)^{-\beta },\;0\leq r\leq d -\varepsilon,
 $$
 $$
  ( \mathfrak{D }^{\beta }_{d-,\,\varepsilon}f)(Q)=      \frac{ f(Q)}{\beta } \left(\frac{1}{\varepsilon^{\beta }}-\frac{1}{(d -r)^{\beta } }    \right),\; d -\varepsilon <r \leq d .
 \end{equation*}
  Now, we can  define  the directional   fractional derivatives as follows
 \begin{equation*}
 \mathfrak{D }^{\beta }_{0+}f=\lim\limits_{\stackrel{\varepsilon\rightarrow 0}{ (L_{p}) }} \mathfrak{D }^{\beta }_{0+,\varepsilon} f  ,\;
  \mathfrak{D }^{\beta }_{d-}f=\lim\limits_{\stackrel{\varepsilon\rightarrow 0}{ (L_{p}) }} \mathfrak{D }^{\beta }_{d-,\varepsilon} f ,\,1\leq p\leq\infty.
\end{equation*}
The properties of these operators  are  described  in detail in the paper  \cite{firstab_lit:1kukushkin2018}.  We suppose  $\mathfrak{I}^{0}_{0+} =I.$ Nevertheless,   this    fact can be easily proved dy virtue of  the reasonings  corresponding to the one-dimensional case and   given in \cite{firstab_lit:samko1987}. We also consider integral operators with a weighted factor (see \cite[p.175]{firstab_lit:samko1987}) defined by the following formal construction
$$
 \left(\mathfrak{I}^{\beta }_{0+}\xi f\right) (Q  ):=\frac{1}{\Gamma(\beta )} \int\limits^{r}_{0}
 \frac{(\xi f) (P+t\mathbf{e})}{( r-t)^{1-\beta }}\left(\frac{t}{r}\right)^{m-1}\!\!\!\!dt,
$$
where $\xi$ is a real-valued  function.

Consider a linear combination of the uniformly elliptic operator, which is written in the divergence form, and
  a composition of a   fractional integro-differential  operator, where the fractional  differential operator is understood as the adjoint  operator  regarding  the Kipriyanov operator  (see  \cite{firstab_lit:kipriyanov1960},\cite{firstab_lit:1kipriyanov1960},\cite{kukushkin2019})
\begin{equation*}
 L  :=-  \mathcal{T}  \, +\mathfrak{I}^{\sigma}_{ 0+}\xi\, \mathfrak{D}  ^{ \beta  }_{d-},
\; \sigma\in[0,1) ,
 $$
 $$
   \mathrm{D}( L )  =H^{2}(\Omega)\cap H^{1}_{0}( \Omega ),
  \end{equation*}
where
$\,\mathcal{T}:=D_{j} ( a^{ij} D_{i}\cdot),\,i,j=1,2,...,m,$
under    the following  assumptions regarding        coefficients
\begin{equation} \label{12}
     a^{ij}(Q) \in C^{2}(\bar{\Omega}),\, \mathrm{Re} a^{ij}\xi _{i}  \xi _{j}  \geq   \gamma_{a}  |\xi|^{2} ,\,  \gamma_{a}  >0,\,\mathrm{Im }a^{ij}=0 \;(m\geq2),\,
 \xi\in L_{\infty}(\Omega).
\end{equation}
Note that in the one-dimensional case the operator $\mathfrak{I}^{\sigma }_{ 0+} \xi\, \mathfrak{D}  ^{ \beta  }_{d-}$ is reduced to   a  weighted fractional integro-differential operator  composition, which was studied properly  by many researchers
    \cite{firstab_lit:2Dim-Kir}, \cite{firstab_lit:15Erdelyi},  \cite{firstab_lit:9McBride},  \cite{firstab_lit:nakh2003}, more detailed historical review  see  in \cite[p.175]{firstab_lit:samko1987}.  In accordance with Theorem 3 \cite{kukushkin2021a}, we claim  that the hypotheses $\mathrm{H}1,\mathrm{H}2$ are fulfilled if $\gamma_{a}$ is sufficiently large in comparison with $\|\xi\|_{\infty},$  where we put $\mathfrak{M}:=C_{0}^{\infty}(\Omega).$  Note that the order $\mu$ of the operator $H$ can be evaluated easily  through  the order of the regular differential operator   and since the latter can be found by methods described in \cite{firstab_lit:Rosenblum}. More precisely, we have
$$
C(\mathfrak{Re}\mathcal{T}f,f)_{\mathfrak{H}} \leq(Hf,f)_{\mathfrak{H}}\leq C (\mathfrak{Re}\mathcal{T}f,f)_{\mathfrak{H}},\,f\in C_{0}^{\infty}(\Omega).
$$
Applying the minimax principle, we get $\lambda_{j}(H)\asymp \lambda_{j}(\mathfrak{Re}\mathcal{T}).$
Using  the well-known formula for regular differential operators   $\lambda_{j}(\mathfrak{Re}\mathcal{T})\asymp j^{2/m}$ (see  \cite{firstab_lit:Rosenblum}), we get  $\mu(H)=2/m.$ Therefore, if we assume that $2\leq m<n,$ then in accordance with Theorem \ref{T4},   we can claim that  there exists a solution of   problem \eqref{1}
where $W$ is a restriction of $L$ on the set $C_{0}^{\infty}(\Omega),$  the coefficients \eqref{12}   are sufficiently smooth to guaranty  the fact  the right-hand side of \eqref{1}   has a sense. Note that the solvability of the   uniqueness problem as well as the opportunity to extend the initial condition depends on the accretive property of the operator $\tilde{W}^{n}.$ The latter problem can be studied by the methods similar to the ones used in the previous paragraph. Indeed, we have established the accretive property in the one-dimensional case.     \\

\noindent{\bf Riesz potential}\\

Consider a   space $L_{2}(\Omega),\,\Omega:=(-\infty,\infty).$      We denote by $H^{2, \varsigma}_{0}(\Omega)$ the completion of the set  $C^{\infty}_{0}(\Omega)$  with the norm
$$
\|f\|_{H^{2,\varsigma}_{0}(\Omega)}=\left\{\|f\|^{2}_{L_{2}(\Omega) }+\|D^{2}f \|^{2}_{L_{2}(\Omega,\omega^{\varsigma})} \right\}^{1/2},\,  \varsigma\in \mathbb{R},
$$
where $\omega(x):=  (1+|x|).$
Let us notice the following fact  (see Theorem 1 \cite{firstab_lit:1Adams}), if $\varsigma>4,$ then
$
H^{2,\varsigma}_{0}(\Omega)\subset\subset L_{2}(\Omega).
$
Consider a Riesz potential
$$
I^{\beta}f(x)=B_{\beta}\int\limits_{-\infty}^{\infty}f (s)|s-x|^{\beta-1} ds,\,B_{\beta}=\frac{1}{2\Gamma(\beta)  \cos  \beta \pi / 2   },\,\beta\in (0,1),
$$
where $f$ is in $L_{p}(\Omega),\,1\leq p<1/\beta.$
It is  obvious that
$
I^{\beta}f= B_{\beta}\Gamma(\beta) (I^{\beta}_{+}f+I^{\beta}_{-}f),
$
where
$$
I^{\beta}_{\pm}f(x)=\frac{1}{\Gamma(\beta)}\int\limits_{0}^{\infty}f (s\mp x) s ^{\beta-1} ds,
$$
the last operators are known as fractional integrals on the  whole  real axis   (see \cite[p.94]{firstab_lit:samko1987}).
 Following the idea of the   monograph \cite[p.176]{firstab_lit:samko1987}
 consider a sum of a differential operator and  a composition of    fractional integro-differential operators
\begin{equation*}
 W   :=  D^{2} a  D^{2}  +   I^{2(1-\beta)}D^{2}+\delta,\;\mathrm{D}(W)=C^{\infty}_{0}(\Omega),\;3/4<\beta<1,
\end{equation*}
where
$
 a(x)\in L_{\infty}(\Omega)\cap C^{ 2  }( \Omega ),\, \mathrm{Re}\,a(x) >\gamma_{a}(1+|x|)^{5},\;\gamma_{a}, \delta >C_{\beta}.
$
Let $\Omega':=[0,\infty),$ consider the functions  $u(t,x),\,t\in \Omega',\,x\in \Omega.$ Similarly to the previous paragraph, we will consider functional spaces with respect to the variable $x$ and we will assume that if $u$ belongs to a functional space then this  fact  holds for all values of the variable $t,$ wherewith all standard topological  properties of a space as completeness, compactness e.t.c. remain correctly defined.
Consider a Cauchy problem \eqref{1}  in the corresponding terms,  under the additional assumptions  $a(x)\in C^{ 2n  }( \Omega ),\,a^{(i)}(x)\in L_{\infty}(\Omega),\,i=1,2,...,2n.$
 Notice that in accordance with the results  \cite{kukushkin2021a}, we claim  that the  hypothesis $\mathrm{H}1,\mathrm{H}2$ hold  regarding: the operator $\tilde{W} ,$ the set  $C^{\infty}_{0}(\Omega),$ the spaces $L_{2}(I),H^{2,\,5}_{0}(\Omega),$   more precisely we should put
$
\mathfrak{H}:=L_{2}(\Omega),\,\mathfrak{H}_{+}:=H^{2,\,5}_{0}(\Omega),\,              \mathfrak{M}:=C^{\infty}_{0}(\Omega).
$
Thus, in accordance with the hypothesis $\mathrm{H}2,$ we have
$$
(H f,f)_{L_{2}}=\mathrm{Re}( W  f,f)_{L_{2}}\geq \gamma_{a}\|f\|^{2}_{H^{2,5}_{0}}= C( D^{2}wD^{2} f,f)_{L_{2}}+C( f,f)_{L_{2}},\,f\in C^{\infty}_{0}(\Omega).
$$
where  $w(x)=(1+|x|)^{5}.$ Let us consider the operator $B=D^{2}wD^{2}+I, \mathrm{D}(B)=C^{\infty}_{0}(\Omega)$ it is clear that by virtue of the minimax principle, we can estimate the eigenvalues of the operator $\tilde{W}$ via estimating the eigenvalues of the operator $\tilde{B}.$ Hence, we have  come to the problem of estimating the eigenvalues of the singular  operator. Here, we should point out that there exists the Fefferman concept that covers such a kind of problems. For instance, the Rozenblyum result is presented in the monograph
\cite[p.47]{firstab_lit:Rosenblum}, in accordance with which we can chose such an unbounded subset of $\mathbb{R}$ that the relation $\lambda_{j}(\tilde{B})\asymp j^{4}$ holds.   Thus, we left this question to the reader for  a more detailed  study    and reasonably allow ourselves to  assume that the  condition $\mu(H)=4$ holds. In this case,  in accordance with
 Theorem \ref{T4}, we are able to present a solution of  the problem \eqref{1}  in the form
$$
 u(t)=\frac{1}{2 \pi i} \int\limits_{\gamma}e^{-\lambda^{n}t}  A(E-\lambda A)^{-1}hd\lambda.
$$
 Now assume additionally that $\mathrm{Im} a=0,$ then $D^{2} a  D^{2}$ is selfadjoint. It follows that $\tilde{W}$ is selfadjoint   and we can easily prove that
$
\mathrm{Re}(\tilde{W}^{n} f,f)_{\mathfrak{H}}\geq0,\;n=1,2,...,\, .
$
Therefore,   applying Theorem \ref{T1},  we can assume  that $h\in \mathfrak{H}$ and claim that the existing solution is unique.
Thus, we have established  the existence and uniqueness of the solution  of  problem  \eqref{1}. Having known the root vectors of the operator  $A,$ applying  formula \eqref{2}, we can represent the obtained solution as a series. \\

\vspace{0.5cm}

\noindent{\bf Difference operator}\\

The approach implemented in studying the difference  operator is remarkable due to the appeared opportunity to set   the problem within the framework of the created theory,  having constructed a  suitable perturbation of the   operator composition.
    Consider a difference operator and its adjoint operator
$$
Jf(x)=c[f(x)-f(x-d)],\,J^{\ast}f(x)=c[f(x)-f(x+d)],\,f\in L_{2}(\Omega),\,\Omega=(-\infty,\infty),\;c,d>0.
$$
Let us find a representation for fractional powers of the operator $A.$ Using formula  (45) \cite{kukushkin2021a}, we get
$$
   J^{\beta}f=\sum\limits_{k=0}^{\infty}C_{k}f(x-kd), \,f\in L_{2}(\Omega),
   \,C_{k}= -\frac{\beta\Gamma(k -\beta)}{k!\Gamma(1 -\beta)}c^{\,\beta},\,\beta\in (0,1).
   $$
 We need the following theorem (see Theorem 5 \cite{kukushkin2021a}).\\

\noindent{\bf Theorem.}   {\it Assume that  $Q$ is a   closed operator acting in $L_{2}(\Omega),\,Q^{-1}\in \mathfrak{S}_{\!\infty}(L_{2}),$ the operator $N$ is strictly accretive, bounded, $\mathrm{R}(Q)\subset \mathrm{D}(N).$ Then
a perturbation
$$
L:= J^{\ast}\!a J+b J^{\beta}+ Q^{\ast}N Q ,\;a,b\in L_{\infty}(\Omega),
$$
   satisfies conditions  H1--H2, if  $\gamma_{N}>\sigma\|Q^{-1}\|^{2},$
where we put $\mathfrak{M}:=\mathrm{D}_{0}(Q),$
$$
 \sigma= 4c\|a\|_{L_{\infty}}+\|b\|_{L_{\infty}}\frac{\beta c^{\,\beta}   }{\Gamma(1 -\beta)}
 \sum\limits_{k=0}^{\infty}\frac{ \Gamma(k -\beta)}{k! }.
$$
}
Observe that by virtue of the made assumptions regarding   $Q,$ we have $\mathfrak{H}_{Q}\subset\subset L_{2}(\Omega).$ We have chosen  the space  $L_{2}(\Omega)$ as a space $\mathfrak{H}$ and the space  $\mathfrak{H}_{Q} $ as a space $\mathfrak{H}_{+}.$   Applying  the condition   $\mathrm{H}2,$   we get
$$
C (Q^{\ast}N Qf,f)_{\mathfrak{H}}\leq(Hf,f)_{\mathfrak{H}}\leq C (Q^{\ast}N Qf,f)_{\mathfrak{H}},\;f\in \mathrm{D}_{0}(Q),
$$
where $H$ is a real part of $\tilde{W}.$ Therefore, by virtue of the minimax principle, we get $\lambda_{j}(H)\asymp \lambda_{j}(Q^{\ast}N Q).$
Hence $\mu(H)=\mu(Q^{\ast}N Q).$
Thus, we have naturally  come to the significance  of the operator $Q$ and the remarkable fact   that we can fulfill the conditions of  Theorem  \ref{T4} chousing the operator $Q$ in the artificial way.
  Applying Theorem  \ref{T4}, we can claim that  there exists a solution of   problem \eqref{1},
where $W$ is a restriction of $L$ on the set $\mathfrak{M}$ (see introduction), functions $a,b$ are sufficiently smooth to guaranty  the fact  the right-hand side of \eqref{1} has a sense. The extension  of the initial conditions on the whole space $\mathfrak{H},$ as well as  solvability of the  uniqueness problem can be implemented in the case  when the operator $\tilde{W}^{n}$  is accretive. In its own turn, it is clear that the particular methods,   to establish the  accretive property, can differ and may depend on the concrete form of the operator $Q.$\\

\noindent{\bf   Artificially constructed normal operator }\\

In this paragraph we consider an operator class  which cannot be completely  studied by methods  \cite{firstab_lit:1Lidskii}, at the same time Theorem \ref{T3} gives us a rather relevant result. Our aim is to construct a normal operator $N$ being satisfied the Theorem \ref{T3} conditions, such that  $N\in \tilde{\mathfrak{S}}_{\alpha} \setminus \mathfrak{S}_{\alpha}$.  Let us consider the following example as a perquisite for the further reasonings.
\begin{ex}\label{E1} Here we would like to produce an example of the sequence   $\{\mu_{n} \}_{1}^{\infty}$  that satisfies the condition
 $
 (\ln^{\kappa+1}x)'_{\mu_{n} }  =o(  n^{-\kappa}   ),\,(0<\kappa<1),\,
$
and at the same time
$$
\sum\limits_{n=1}^{\infty}\frac{1}{|\mu_{n}|^{1/\kappa}}=\infty.
$$
Consider a sequence $\mu_{n}=n^{\kappa}\ln^{\kappa} n \cdot \ln^{\kappa}\ln n,$ then using the integral test for convergence we can easily see that the previous series   diverges. At the same time substituting, we get
$$
\frac{\ln^{\kappa}\mu_{n}}{\mu_{n}} \leq \frac{ C }{n^{\kappa}     \ln^{\kappa}\ln n},\;n=1,2,...\,,
$$
what gives us the fulfilment of the first condition.
\end{ex}
Consider the abstract separable Hilbert space $\mathfrak{H}$ and an operator $N$ acting in the space as follows
$$
Nf=\sum\limits_{n=1}^{\infty}\lambda_{n}f_{n}e_{n},\;f_{n}=(f,e_{n})_{\mathfrak{H}},\;
 \lambda_{n}=\mu_{n}+i \eta_{n}
$$
where $\{e_{n}\}_{1}^{\infty}\subset \mathfrak{H}$ is an orthonormal basis, the sequence  $\{\mu_{n}\}_{1}^{\infty}$  is defined in Example \ref{E1}, $|\eta_{n}|<M\mu_{n},\,M>0,\; n=1,2,...\,.$
Define the space $\mathfrak{H}_{+}$ as follows
$$
\mathfrak{H}_{+}:=\left\{f\in \mathfrak{H}:\;\|f\|^{2}_{\mathfrak{H}_{+}}:=\sum\limits_{n=1}^{\infty}|\lambda_{n}||f_{n}|^{2}<\infty\right\}.
$$
It is clear that $\mathfrak{H}_{+}$ is dense in $\mathfrak{H},$ since  $\{e_{n}\}_{1}^{\infty}\subset \mathfrak{H}_{+}.$ Let us show that embedding of the spaces $\mathfrak{H}_{+}\subset \mathfrak{H}$ is compact. Consider the operator $B: \mathfrak{H}\rightarrow \mathfrak{H}$ defined as follows
$$
Bf=\sum\limits_{n=1}^{\infty}|\lambda_{n}|^{-1/2}f_{n}e_{n}.
$$
Note that compactness of the operator $B$ can be proved easily due to the well-known  criterion of compactness in the Banach space endowed with a basis (we left the prove  to the reader).
Notice that if $f\in \mathfrak{H}_{+},$ then $g\in\mathfrak{H},$  where $g$ is defined by its fourier coefficients    $g_{n}=|\lambda_{n}|^{1/2}|f_{n}|.$ By virtue of such a correspondence we can consider any bounded set in the space $\mathfrak{H}_{+}$ as a bounded set in the space $\mathfrak{H}.$ Applying the operator $B$ to the element $g,$ we get the element $f.$ Due to the compactness of the operator $B$ we can conclude that the image of the bounded set in the sense of the norm $\mathfrak{H}_{+}$  is a compact set in the sense of the norm $\mathfrak{H}.$
Define the set $\mathfrak{M}$ as a linear manifold generated by  the basis vectors.  Thus, we have obtained the relation $\mathfrak{H}_{+}\subset\subset\mathfrak{H}$ and established the fulfilment of the hypotheses $\mathrm{H}1.$     The first relation of the hypotheses $\mathrm{H}2$ can be obtained  easily due to the application of the Cauchy-Schwarz inequality. To obtain the second one consider
$$
\mathrm{Re}(Nf,f)_{\mathfrak{H}}=\sum\limits_{n=1}^{\infty}\mathrm{Re} \lambda_{n} |f_{n}|^{2} \geq  \left(1+M^{2}\right)^{-1/2}  \sum\limits_{n=1}^{\infty}  |\lambda_{n}| |f_{n}|^{2}= \left(1+M^{2}\right)^{-1/2} \|f\|^{2}_{\mathfrak{H}_{+}}.
$$
Now, we  conclude that hypotheses $\mathrm{H}1, \mathrm{H}2$ hold. Consider a Cauchy problem
\begin{equation}\label{13}
\frac{du}{dt}=N^{1/\kappa}u,\;u(0)=h\in \mathrm{D}(N),\;\kappa= 1/2,1/3,...\,,
\end{equation}
where $h$ is supposed to be an  arbitrary element if the operator $N^{1/\kappa}$ is accretive.
In accordance with Theorem \ref{T4}, we conclude that there exists a solution of  problem \eqref{13} presented by the series \eqref{2}.
 Moreover, we claim that  under the assumption  $ M\leq \tan\{\pi \kappa/2 \},$ the existing solution is unique and we can extend the initial condition assuming that $h\in \mathfrak{H}.$ For this purpose, in accordance with Theorem \ref{T3}, let us prove that
$
\mathrm{Re}(N^{1/\kappa}f,f)_{\mathfrak{H}}\geq0.
$
 The latter fact  follows from
the relation
$$
\mathrm{Re }\lambda_{n}^{1/\kappa}=|\lambda_{n}|^{1/\kappa} \cos   \left(\frac{\mathrm{arg}\lambda_{n}}{\kappa}   \right) \geq |\lambda|^{1/\kappa} \cos   \frac{\pi}{2}\,,
$$
and the representation
$$
\mathrm{Re}(N^{1/\kappa}f,f)_{\mathfrak{H}}= \sum\limits_{n=1}^{\infty}\mathrm{Re}\, \lambda^{1/\kappa}_{n} |f_{n}|^{2}.
$$
Thus, we obtain the desired result. Note that the constructed normal operator indicates the significance  of the made  in Theorem \ref{T3} clarification  of the results \cite{firstab_lit:1Lidskii}. However, we produce one more relevant application of the mentioned  theorem in the following paragraph.  \\

\noindent{\bf  Evolution equations with the fractional derivative at the left-hand side}\\

   In this paragraph, we still consider   a Hilbert space $\mathfrak{H}$ consists of   element-functions $u:\mathbb{R}_{+}\rightarrow \mathfrak{H},\,u:=u(t),\,t\geq0$    assuming  that if $u$ belongs to $\mathfrak{H}$    then the fact  holds for all values of the variable $t.$   We understand such operations as differentiation and integration in the generalized sense that is caused by the topology of the Hilbert space $\mathfrak{H}.$ The derivative is understood as a    limit
$$
  \frac{u(t+\Delta t)-u(t)}{\Delta t}\stackrel{\mathfrak{H}}{ \longrightarrow}\frac{du}{dt} ,\,\Delta t\rightarrow 0.
$$
Let $t\in  I:=[a,b],\,0< a <b<\infty.$ The following integral is understood in the Riemann  sense as a limit of partial sums
\begin{equation*}
\sum\limits_{i=0}^{n}u(\xi_{i})\Delta t_{i}  \stackrel{\mathfrak{H}}{ \longrightarrow}  \int\limits_{ I }u(t)dt,\,\zeta\rightarrow 0,
\end{equation*}
where $(a=t_{0}<t_{1}<...<t_{n}=b)$ is an arbitrary splitting of the segment $ I ,\;\zeta:=\max\limits_{i}(t_{i+1}-t_{i}),\;\xi_{i}$ is an arbitrary point belonging to $[t_{i},t_{i+1}].$
The sufficient condition of the last integral existence is a continuous property (see\cite[p.248]{firstab_lit:Krasnoselskii M.A.}) i.e.
$
u(t)\stackrel{\mathfrak{H}}{ \longrightarrow}u(t_{0}),\,t\rightarrow t_{0},\;\forall t_{0}\in  I.
$
The improper integral is understood as a limit
\begin{equation*}
 \int\limits_{a}^{b}u(t)dt\stackrel{\mathfrak{H}}{ \longrightarrow} \int\limits_{a}^{c}u(t)dt,\,b\rightarrow c,\,c\in  [0,\infty].
\end{equation*}
Combining the operations we can consider a generalized  fractional derivative
  in the Riemann-Liouville sense (see \cite{firstab_lit:samko1987}),     in the formal form, we have
$$
   \mathfrak{D}^{1/\alpha}_{-}f(t):=-\frac{1}{\Gamma(1-1/\alpha)}\frac{d}{d t}\int\limits_{0}^{\infty}f(t+x)x^{-1/\alpha}dx,\;\alpha>1.
$$
Let us study   a Cauchy problem
\begin{equation}\label{14}
   \mathfrak{D}^{1/\alpha}_{-}  u=\tilde{W}u ,\;u(0)=h\in \mathrm{D}(\tilde{W}),
\end{equation}
where in the case  when the operator composition $\mathfrak{D}^{1-1/\alpha}_{-}\tilde{W}$ is   accretive,  we assume  that   $h\in \mathfrak{H}.$
 \begin{teo}\label{T5}
Assume that  the Theorem \ref{T3} conditions  hold,   then
there exists a solution of the Cauchy problem \eqref{14} in the form
\begin{equation}\label{15}
u(t)= \frac{1}{2\pi i}\int\limits_{\gamma}e^{-\lambda^{\alpha}t}A(I-\lambda A)^{-1}h d \lambda
=  \sum\limits_{\nu=0}^{\infty}  \sum\limits_{q=N_{\nu}+1}^{N_{\nu+1}}\sum\limits_{\xi=1}^{m(q)}\sum\limits_{i=0}^{k(q_{\xi})}e_{q_{\xi}+i}c_{q_{\xi}+i}(t),
\end{equation}
where
\begin{equation*}
  \sum\limits_{\nu=0}^{\infty}\left\|\sum\limits_{q=N_{\nu}+1}^{N_{\nu+1}}\sum\limits_{\xi=1}^{m(q)}\sum\limits_{i=0}^{k(q_{\xi})}e_{q_{\xi}+i}c_{q_{\xi}+i}(t)\right\|_{\mathfrak{H}}<\infty,
\end{equation*}
a sequence of natural numbers $\{N_{\nu}\}_{0}^{\infty}$ can be chosen in accordance with the claim of    Theorem \ref{T3}.
Moreover, the existing solution is unique if the operator composition  $\mathfrak{D}^{1-1/\alpha}_{-}\tilde{W}$ is accretive.
\end{teo}
\begin{proof}
  Let us find a solution of problem \eqref{14} in the form \eqref{15}.
  Analogously to the reasonings of Theorem \ref{T4},
using Lemma 6 \cite{firstab_lit:1kukushkin2021}, it is not hard to prove  that the following  integrals exist  i.e.
\begin{equation}\label{17}
 \frac{1}{2\pi i} \int\limits_{\gamma}e^{-\lambda^{\alpha}t} (E-\lambda A)^{-1}h d\lambda  \in \mathfrak{H};\;\frac{d u}{d t}=-\frac{1}{2\pi i} \int\limits_{\gamma}e^{-\lambda^{\alpha}t}  \lambda^{\alpha } A(E-\lambda A)^{-1}h \,d\lambda \in \mathfrak{H}.
\end{equation}
Note that the first relation gives us the fact
   $u(t)\in \mathrm{D}(\tilde{W}).$
 Applying  the scheme of the proof corresponding to the  ordinary integral calculus,  using the contour $\gamma_{k}$ (see reasonings of Theorem \ref{T4}),   applying  Lemma   6  \cite{firstab_lit:1kukushkin2021}, we can establish a  formula
\begin{equation}\label{18}
\int\limits_{0}^{\infty}x^{-\sigma}dx\int\limits_{\gamma}e^{-\lambda^{\zeta}(t+x)}\lambda^{m}A(E-\lambda A)^{-1}h d\lambda=\int\limits_{\gamma}e^{-\lambda^{\zeta}t}\lambda^{m}A(E-\lambda A)^{-1}h d\lambda\int\limits_{0}^{\infty}x^{-\sigma}e^{-\lambda^{\zeta}x}dx,
\end{equation}
where $\sigma\in(0,1),\,\zeta>0,\,m\in \mathbb{N}_{0}.$
In the same way, we get
$$
\frac{d}{dt}\int\limits_{\gamma}\lambda^{1-\alpha}e^{-\lambda^{\alpha}t}A(E-\lambda A)^{-1}h d\lambda =
-\int\limits_{\gamma}\lambda e^{-\lambda^{\alpha}t}A(E-\lambda A)^{-1}h d\lambda.
$$
Applying  the obtained  formulas, taking into account a relation
$$
  \int\limits_{0}^{\infty}x^{-1/\alpha}e^{-\lambda^{\alpha}x}dx=  \Gamma(1-1/\alpha) \lambda^{1-\alpha},
$$
we get
\begin{equation}\label{19}
\mathfrak{D}^{1/\alpha}_{-}u=\frac{1}{2\pi i} \int\limits_{\gamma}e^{-\lambda^{\alpha}t}  \lambda  A(E-\lambda A)^{-1}h d\lambda.
\end{equation}
Making the substitution using the formula  $\lambda A(E-\lambda A)^{-1}=(E-\lambda A)^{-1}-E,$     we obtain
$$
  \mathfrak{D}^{1/\alpha}_{-}u= \frac{1}{2\pi i} \int\limits_{\gamma}e^{-\lambda^{\alpha}t}   (E-\lambda A)^{-1}h\, d\lambda-\frac{1}{2\pi i} \int\limits_{\gamma}e^{-\lambda^{\alpha}t}   h \,d\lambda=I_{1}+I_{2}.
$$
The second integral equals zero by virtue  of  the fact  that the   function under the integral is analytical inside the intersection of the  domain $G$  with the circle of  an arbitrary radius $R$   and it   decreases sufficiently fast on the arch of the radius $R,$ when $R\rightarrow\infty,$  here we denote  by $G$ the interior of the contour $\gamma.$  Now, if we consider the    expression for $u,$ we obtain the fact that $u$ is a solution of the equation i.e.
$
\mathfrak{D}^{1/\alpha}_{-}u=\tilde{W }u.
$
 We obtain  the decomposition   on  the series  of the root vectors \eqref{15}   due to  Theorem   \ref{T3}.
 Let us show that the initial condition holds in the sense
$
u(t)   \xrightarrow[   ]{\mathfrak{H}}  h,\,t\rightarrow+0.
$
It becomes clear if $h\in \mathrm{D}(\tilde{W})$ for in this case  it suffices to apply Lemma 7 \cite{firstab_lit:1kukushkin2021},     what  gives  us the desired result. To establish the fact in  the case   $h\in \mathfrak{H},$ we should   involve the accretive property of the operator composition  $ \mathfrak{D}^{1-1/\alpha}_{-}\tilde{W}  .$
Following the scheme of the Theorem \ref{T4} reasonings, consider an operator
$$
S_{t}h=\frac{1}{2\pi i} \int\limits_{\gamma}e^{-\lambda^{\alpha}t}A(E-\lambda A)^{-1}h\, d\lambda,\,t>0.
$$
Using the first relation \eqref{17}, the fact $A\in \mathcal{B}(\mathfrak{H}),$ we can easily prove that   $S_{t}:\mathfrak{H}\rightarrow \mathfrak{H}.$
Let us prove  that
$
\|S_{t}\|_{\mathfrak{H}\rightarrow \mathfrak{H}}\leq1,\;t>0.
$
We need to establish the following  formula
\begin{equation}\label{20}
\mathfrak{D}^{1-1/\alpha}_{-}\mathfrak{D}^{1/\alpha}_{-}u=-u'.
\end{equation}
Using relation \eqref{18}, we get
\begin{equation*}
\int\limits_{0}^{\infty}x^{1/\alpha-1}dx\int\limits_{\gamma}e^{-\lambda^{\alpha}(t+x)}\lambda A(E-\lambda A)^{-1}h d\lambda=\int\limits_{\gamma}e^{-\lambda^{\alpha}t}\lambda A(E-\lambda A)^{-1}h d\lambda\int\limits_{0}^{\infty}x^{1/\alpha-1}e^{-\lambda^{\alpha}x}dx=
\end{equation*}
$$
=\Gamma(1/\alpha)\int\limits_{\gamma}e^{-\lambda^{\alpha}t}  A(E-\lambda A)^{-1}h d\lambda = 2\pi i \Gamma(1/\alpha) u(t).
$$
Substituting  \eqref{19} to the first term of the latter  relation, using the given above definition of the generalized fractional derivative, we obtain \eqref{20}. Now, assume that   $h\in \mathrm{D}(\tilde{W}).$
Applying  the operator $\mathfrak{D}^{1-1/\alpha}_{-}$ to   both sides of relation \eqref{14}, using   formula \eqref{20},
we get
$
u'+\mathfrak{D}^{1-1/\alpha}_{-}\tilde{W}u=0.
$
 Let us multiply the both sides of the latter  relation  on $u$  in the sense of the inner product, we get
$
\left(u'_{t},u\right)_{\mathfrak{H}}+(\mathfrak{D}^{1-1/\alpha}_{-}\tilde{W}u,u)_{\mathfrak{H}}=0.
$
Consider a real part of the latter relation, we have
$
\mathrm{Re}\left(u'_{t},u\right)_{\mathfrak{H}}+\mathrm{Re}(\mathfrak{D}^{1-1/\alpha}_{-}\tilde{W}u,u)_{\mathfrak{H}}= \left(u'_{t},u\right)_{\mathfrak{H}}/2+ \left(u, u'_{t}\right)_{\mathfrak{H}}/2+\mathrm{Re}(\mathfrak{D}^{1-1/\alpha}_{-}\tilde{W}u,u)_{\mathfrak{H}}.
$
Therefore
$
   \left(\|u(t)\|_{\mathfrak{H}}^{2}\right)'_{t}  =-2\mathrm{Re}(\mathfrak{D}^{1-1/\alpha}_{-}\tilde{W}u,u)_{\mathfrak{H}}\leq 0.
$
Integrating both sides, we get
$$
  \|u(\tau)\|_{\mathfrak{H}}^{2}-  \|u(0)\|_{\mathfrak{H}}^{2}=\int\limits_{0}^{\tau} \frac{d }{dt}\|u(t)\|_{\mathfrak{H}}^{2} dt\leq 0.
$$
The last relation can be rewritten in the form
$
\|S_{t}h\|_{\mathfrak{H}}\leq \|h\|_{\mathfrak{H}},\,h\in \mathrm{D}(\tilde{W}).
$
Since $\mathrm{D}(\tilde{W})$ is a dense set in $ \mathfrak{H},$ then we obviously  obtain  the  desired result i.e. $\|S_{t}\|_{\mathfrak{H}\rightarrow \mathfrak{H}}\leq 1.$  The rest part of the proof is absolutely analogous to the proof of Theorem \ref{T4}.   The decomposition  on  the series  of the root vectors \eqref{15} is given  by virtue  of Theorem   \ref{T3}.
 The uniqueness follows  from the accretive property of the  operator composition  $\mathfrak{D}^{1-1/\alpha}_{-}\tilde{W}.$
\end{proof}

\begin{remark}
Eventually, we want to note  that  Theorem \ref{T5} is applicable to the   operators studied  in the previous paragraphs. The question that may appear is related to the description of the operator class for which the considered above operator  composition is accretive.
\end{remark}

\newpage

\section{Conclusions}

In this paper we studied the Cauchy problems for the evolution equation in the abstract Hilbert space. The made approach allows us to obtain a solution  analytically  for the right-hand side belonging to a sufficiently wide class of operators. In this regard  such operators as the Riemann-Liouville fractional differential operator, the Kipriyanov operator, the Riesz potential, the difference operator have been involved. Moreover, we produced the artificially constructed normal operator for which the clarification of the Lidskii V.B. results  relevantly works. It is remarkable that Theorem \ref{T5} covers many previously obtained results in the framework of the theory of factional differential equations. However, the main advantage is the obtained abstract formula for the solution. Moreover,   the norm-convergence of the series representing the solution allows us to apply the methods of the approximation theory. It is also worth noticing that Theorem \ref{T5} jointly with Theorem \ref{T3} give  us the opportunity to minimize conditions imposed on the fractional order at the left-hand side of the equation  and   the artificially constructed normal operator provides the relevance and significance of such an achievement. Here, we should admit that  the analogs of  Theorem \ref{T5} can be obtained by  Lidskii V.B. methods \cite{firstab_lit:1Lidskii}, but the used  in the paper \cite{firstab_lit:1kukushkin2021}  variant of the natural  lacunae method   allows us to formulate the optimal conditions in comparison with the Lidskii V.B. results.

The interesting  problem  that may appear is how to consider the evolution equations of the arbitrary  real order.
  We should note that the applied technique can be relevant in the case of the integer  order higher than one for   the operator function   under the integral presenting the solution is sufficiently "good" what allows us to differentiate it. The latter opportunity becomes more valuable  due to the appeared formula connecting the derivative of the integer order and  the corresponding power of the operator. This creates a prerequisite to consider a generalized  Taylor series and raise a question on convergence at least and representing the solution at most. It is remarkable that the supposed  correspondence between derivatives and  the operator powers leads us to the study of   series of   unbounded operators. Apparently,  the hypotheses on the representing the solution by the  Taylor series is equivalent to existing   the fixed point of the mentioned operator series.
 We hope that  this  idea creates a prerequisite for further study in the direction.

\end{document}